\documentclass[11pt]{amsart}
\usepackage{amsmath, amssymb, amsthm, amscd, amssymb, color, latexsym, mhsetup, mathtools, tikz}
\usepackage[margin=1in]{geometry}
\usepackage[shortlabels]{enumitem}
\usetikzlibrary{decorations.markings, positioning,chains,fit,shapes,calc}

\newtheorem{theorem}{Theorem}[section]
\newtheorem{proposition}[theorem]{Proposition}
\newtheorem{lemma}[theorem]{Lemma}

\newtheorem*{question*}{Question}
\newtheorem{remark}[theorem]{Remark}

\theoremstyle{definition}
\newtheorem{definition}[theorem]{Definition}
\newtheorem*{definition*}{Definition}
\newtheorem{example}[theorem]{Example}

\theoremstyle{remark}
\newtheorem*{remark*}{Remark}

\title{Upper tails of subgraph counts in directed random graphs}
\author{Jiyun Park}

\begin{document}

\begin{abstract}
    The upper tail problem in a sparse Erd\H{o}s-R\'enyi graph asks for the probability that the number of copies of some fixed subgraph exceeds its expected value by a constant factor. We study the analogous problem for oriented subgraphs in directed random graphs. By adapting the proof of Cook, Dembo, and Pham \cite[Theorem~1.1]{CDP2023}, we reduce this upper tail problem to the asymptotic of a certain variational problem over edge weighted directed graphs. We give upper and lower bounds for the solution to the corresponding variational problem, which differ by a constant factor of at most $2$. We provide a host of subgraphs where the upper and lower bounds coincide, giving the solution to the upper tail problem. Examples of such digraphs include triangles, stars, directed $k$-cycles, and balanced digraphs.
\end{abstract}

\maketitle

\section{Introduction}\label{sec:intro}

\subsection{The upper tail problem in random graphs}

The Erd\H{o}s-R\'enyi random graph $G(n, p)$ is defined as a graph on $n$ vertices where each edge is chosen independently with probability $p$. Given a graph $H$, what is the probability that the number of copies of $H$ in $G(n, p)$ exceeds its expectation by a constant factor? This problem, sometimes referred to as the ``infamous upper tail'' problem, has been studied extensively in recent years. Like many other upper tail problems, the problem is broken down into two steps:

\begin{enumerate}[(a)]
    \item Formulate a large deviation framework, reducing the problem to a variational problem over edge-weights.
    \item Solve the variational problem.
\end{enumerate}

For dense graphs (constant $p$), Chatterjee and Varadhan \cite{CV2011} first introduced a large deviation principle. Subsequently, Chatterjee and Dembo \cite{CD2016} showed a LDP-type statement in the sparse regime, where $n^{- \kappa(H)} \le p \le 1$ for some constant $\kappa(H) > 0$. This was followed by a series of work devoted to improving this range (see \cite{CD2020,HMS2022,A2020, E2018,BB2023}).

There has also been much progress in solving the variational problem. For dense graphs, Lubetzky and Zhao \cite{LZ2015} determined the regime where the optimum is achieved by a constant graphon for every regular graph $H$. In the sparse setting, Lubetzky and Zhao \cite{LZ2017} solved the variational problem in the case $H = K_3$ for $n^{1 /2} \ll p \ll 1$. The problem for general graphs was resolved by Bhattacharya, Ganguly, Lubetzky, and Zhao \cite{BGLZ2017} for $n^{1/\Delta} \ll p \ll 1$.

Some of these results have been extended to a wider class of random graphs. For instance, Bhattacharya and Dembo \cite{BD2021} developed a large deviation framework for sparse random regular graphs. This result was further improved by Gunby \cite{G2021}, which also studied the corresponding variational problem. Cook and Dembo \cite{cook2022typical} solved the upper tail problem for sparse exponential random graphs and also characterized the conditional structure of the Erd\H{o}s-R\'enyi graph. The large deviation theory for random hypergraphs was developed by Cook, Dembo, and Pham \cite{CDP2023}. The variational problem for hypergraphs was studied by Liu and Zhao \cite{LZ2021}, where they solved the problem for cliques and conjectured a solution for general hypergraphs.

There has also been progress on the lower tail problem, i.e., determining the probability that the number of copies of $H$ is less than $\eta$ times its expected value, where $0 < \eta < 1$. Along with the upper tail problem, Chatterjee and Varadhan \cite{CV2011} established a large deviation principle for the lower tail in dense Erd\H{o}s-R\'enyi graphs. This was extended to the sparse setting by Chatterjee and Dembo \cite{CD2016}, where they showed that the problem can be reduced to a variational problem provided $n^{-\alpha_H} \ll p$ for some $\alpha_H$. Zhao \cite{Z2017} determined that the variational problem has a constant solution when $\eta$ is sufficiently close to $1$ whereas the constant graphon is not the optimizer when $\eta$ is close to $0$.

\subsection{Directed graphs}

In this paper, we aim to derive similar results for the upper tails of directed random graphs. A \emph{directed graph}, or \emph{digraph}, is a graph $\mathsf{H} = (\mathsf{V}, \mathsf{E})$ where each edge $e \in \mathsf{E}$ is an ordered pair of vertices. We shall assume that a digraph contains no self-loops and there exist at most one directed edge for every ordered pair of vertices. The $p$-uniform Erdos-R\'enyi digraph $\mathbf{G}$ on $n$ vertices is defined so that every directed edge between two vertices exists independently with probability $p$. In other words, its incidence matrix is an element of the set
\[
    \mathcal{A}_n := \{ Z \in \mathcal{Z}_n : Z_{i,j} \in \{0, 1\}, Z_{i,i} = 0 \},
\]
where each $Z_{i, j} = 1$ with probability $p$ (if $i \ne j$) and $Z_{i, j} = 0$ otherwise.

While not as extensive as that of undirected graphs, there is still a vast amount of literature on the theory of digraphs. These results have many connections to other branches of mathematics, as well as applications to computer science, operations research, social sciences and engineering. For a broad survey on digraphs, see \cite{JG2009, JG2018} (and the references therein).

\subsection{Main result}

Let $H$ be a digraph with maximum degree $\Delta$, where the degree is the sum of in-degrees and out-degrees of a vertex. The \emph{homomorphism density} of $H$ in a weighted digraph $Q$ is defined as
\[
	t(H, Q) := \frac{1}{n^{\mathsf{v}(H)}} \sum_{\phi : \mathsf{V}(H) \to [n]} \prod_{e \in \mathsf{E}(H)} Q(\phi(e)).
\]
Define the upper tail rate and entropic optimization problem
\begin{align*}
    \mathrm{UT}_{n, p} (H, \delta) &:= - \log \mathbb{P}(t(H, \mathbf{G}) \ge (1 + \delta) p^{\mathsf{e}(H)}), \\
    \Phi_{n, p} (H, \delta) &:= \inf_{Q \in \mathcal{Q}_n} \{ I_p (Q) : t(H, Q) \ge (1 + \delta) p^{\mathsf{e}(H)} \}. \\
\end{align*}
Here, $\mathbf{G}$ is the directed Erd\H{o}s-R\'enyi graph on $n$ vertices, where each directed edge has probability $p = p(n)$ of appearing. $\mathcal{Q}_n$ refers to the collection of fractional digraphs, where each edge has weights in $[0, 1]$. $I_p (\cdot)$ is defined by
\[
    I_p(x) = x \log \frac{x}{p} + (1-x) \log \frac{1-x}{1-p}.
\]
The the analogous lower tail quantities are as follows.
\begin{align*}
    \mathrm{LT}_{n, p} (H, \delta) &:= - \log \mathbb{P}(t(H, \mathbf{G}) \le (1 - \delta) p^{\mathsf{e}(H)}), \\
    \Psi_{n, p} (H, \delta) &:= \inf_{Q \in \mathcal{Q}_n} \{ I_p (Q) : t(H, Q) \le (1 - \delta) p^{\mathsf{e}(H)} \}. \\
\end{align*}

We also define joint quantities as follows. For fixed digraphs $ \underline{H} = (H_1 , \dots , H_m)$ and $\underline{\delta} = (\delta_1 , \dots , \delta_m)$, $\mathrm{UT}_{n, p} (\underline{H}, \underline{\delta})$ denotes the joint upper-tail rate
\[
    \mathrm{UT}_{n, p} (\underline{H}, \underline{\delta}) := - \log \mathbb{P} \left( t(H_k , \mathbf{G}) \ge (1 + \delta_k) p^{\mathsf{e}(H_k)}, 1 \le k \le m \right).
\]
$\mathrm{LT}_{n, p}(\underline{H}, \underline{\delta})$, $\Phi_{n, p}(\underline{H}, \underline{\delta})$, and $\Psi_{n, p}(\underline{H}, \underline{\delta})$ are defined similarly. In Section~\ref{sec:prelim}, we prove the following LDP-type statement for the subgraph count of $H$ in a Erd\H{o}s--R\'enyi random digraph.    

\begin{theorem}[{\cite[Theorem~1.1]{CDP2023}}]
    \label{thm:gen-large-dev}
    Fix digraphs graphs $H_1 , \dots , H_m$. Let $\Delta_{\max} = \max_k \Delta(H_k)$ and $\Delta_{\max}' = \max_k \Delta'(H_k)$ ($\Delta'$ is defined in \cite{CDP2023}).
    \begin{enumerate}[(a)]
        \item Suppose $H_1 , \dots , H_k$ are oriented graphs. If $n^{-1/\Delta_{\max}'} \ll p < 1$, then for any fixed $\delta_1 , \dots , \delta_m > 0$,
        \[
            \mathrm{UT}_{n, p} (\underline{H}, \underline{\delta}) \ge (1 - o(1)) \Phi_{n ,p}(\underline{H}, \underline{\delta} - o(1)).
        \]
        \item If $n^{-1 / \Delta_{\max}} \ll p < 1$, then for any fixed $\delta_1 , \dots , \delta_m > 0$,
        \[
            \mathrm{UT}_{n, p} (\underline{H}, \underline{\delta}) \le (1 + o(1)) \Phi_{n ,p}(\underline{H}, \underline{\delta} + o(1)).
        \]
        \item Suppose $H_1 , \dots , H_k$ are oriented graphs. If $n^{-1 / \Delta_{\max}'} \log n \ll p < 1$, then for any fixed $\delta_1 , \dots , \delta_m \in (0, 1)$,
        \[
            \mathrm{LT}_{n, p} (\underline{H}, \underline{\delta}) = (1 + o(1)) \Psi_{n ,p}(\underline{H}, \underline{\delta} + o(1)).
        \]
    \end{enumerate}
\end{theorem}

Though we have omitted the definition of $\Delta'$, we remark that $\Delta' \le \Delta + 1$. In the case where $\underline{H}$ consists only of stars, this can be improved to $\Delta' = \Delta + 1/2$.

Given the definitions in Section~\ref{sec:lower-bd}, the proof can be done along the same lines as in the one given by Cook, Dembo, and Pham \cite{CDP2023} (and will not be repeated here). We only need to make some minor modifications, such as summing over ordered pairs instead of unordered pairs. One example is in the proof of \cite[Lemma~5.1]{CDP2023}. In our setting, entries of $\boldsymbol{A}$ are no longer symmetric. This means that we need not quotient $[n]^2$ by the symmetric group, but this only simplifies the proof. When following the proofs in \cite{CDP2023}, assume $H$ to be an oriented graph in all sections except \cite[Section~9]{CDP2023}. In \cite[Section~9]{CDP2023}, we may weaken this condition and only assume that $H$ is a digraph.

For parts (a) and (c), It is important that the graphs are oriented. This is because the counting lemma \cite[Theorem~4.1]{CDP2023} fails once we have multiple edges between a pair of vertices. For an example of this in the undirected setting, see \cite[Exercise~10.26]{L2012}. In the case of simple (but not oriented) digraphs, the following example demonstrates a counterexample to part (a).

\begin{example}
    \label{ex:C2}
    Consider $C_2$, the digraph with two vertices and edges in both directions. The arguments in proceeding sections (e.g., Example~\ref{ex:balanced}) show that
    \[
        \Phi_{n, p}(C_2 , \delta) = (1 + o(1)) \delta n^2 p^2 \log(1/p).
    \]
    However, $\hom(C_2, \mathbf{G})$ is simply the binomial distribution $\mathrm{Bin}(\frac{n(n-1)}{2}, p^2)$. By standard arguments, we can show that
    \[
        \mathrm{UT}_{n, p} (C_2 , \delta) = (1 + o(1)) \frac{1}{2}n^2 p^2 [(1 + \delta) \log (1 + \delta) - \delta].
    \]
    This implies $\mathrm{UT}_{n, p} (C_2 , \delta) \ll \Phi_{n, p}(C_2, \delta)$ if $p \ll 1$. 
    
    Another way to view $UT_{n, p}(C_2, \delta)$ is as follows. Instead of choosing each edge independently, consider the random graph $\mathbf{G}'$ where every bidirectional pair is chosen with probability $p^2$. Clearly, $\hom(C_2, \mathbf{G}) = \hom(C_2 , \mathbf{G}')$. However, $\mathbf{G}'$ can be seen as the \emph{undirected} random graph $G(n, p^2)$, where we are now counting subgraphs with two vertices and one edge between them. This means that we can use the large deviation framework already established in the undirected setting to do this analysis.
\end{example}

The remainder of this paper focuses on determining the behavior of $\Phi_{n, p}(H, \delta)$ for a single digraph $H$. We provide two theorems that respectively give upper and lower bounds on this value.

\begin{definition}
    $H^*$ is the induced subgraph of $H$ on all vertices with degree $\Delta$. Let $\mathcal{S}_H$ be the collection of all independent sets of $H^*$. For any $S \in \mathcal{S}_H$, let $T = \mathsf{N}(S)$ be the set of neighbors of $S$. Define $\mathsf{E}(S, T)$ as the set of edges from $S$ to $T$.
    \[
        A_S := |\mathsf{N}^+(S)| = \left| \{ v \in T: (u, v) \in \mathsf{E}(S, T) \text{ for some } u \in S \} \right|
    \]
    \[
        B_S := |\mathsf{N}^-(S)| = \left| \{ v \in T: (v, u) \in \mathsf{E}(T, S) \text{ for some } u \in S \} \right|.
    \]
    Further, define
    \[
        f_H(x_1 , x_2 , y_1 , y_2) := \sum_{S \in \mathcal{S}_H} x_1^{\mathsf{v}^+ (S)} x_2^{\mathsf{v}^- (S)} (x_1 \wedge x_2)^{\mathsf{v}^{\pm} (S)} y_1 ^{A_S} y_2 ^{B_S}
    \]
    and
    \[
        F(H, \delta) := \inf_{\substack{0 \le x_1, x_2 \\ 0 \le y_1 , y_2 \le 1}} \{ x_1 y_1 + x_2 y_2 : f_H(x_1, x_2, y_1, y_2) = 1 + \delta \}.
    \]
\end{definition}

\begin{theorem}[Upper bound of variational problem]\label{thm:upper-bd}
    \quad
    \begin{enumerate}[(a)]
        \item Let $H$ be a digraph with maximum degree $\Delta$ and $\delta > 0$. If $n^{-1/\Delta} \ll p \ll 1$, then
        \[
        \frac{\Phi_{n, p}(H, \delta)}{n^2 p^\Delta \log(1/p)} \le F(H, \delta) + o(1).
        \]
        Further, the infimum $F(H, \delta)$ of $f_H$ is obtained when $y_1 \vee y_2 = 1$.
        \item Let $H$ be a connected, $\Delta$-regular digraph and $\delta > 0$. If $n^{-2/\Delta} \ll p \ll 1$, then
        \[
            \frac{\Phi_{n, p}(H, \delta)}{n^2 p^\Delta \log(1/p)} \le \delta^{2 / \mathsf{v}(H)} + o(1).
        \]
        Combined with part~(a), this implies that
        \[
            \frac{\Phi_{n, p}(H, \delta)}{n^2 p^\Delta \log(1/p)} \le \min(F(H, \delta), \delta^{2 / \mathsf{v}(H)}) + o(1).
        \]
    \end{enumerate}
\end{theorem}

The proof of this theorem will be given in Section~\ref{sec:upper-bd}. We proceed to lower bounds of $\Phi_{n, p}$.
\begin{definition}
    For each independent set $S$ of $H^*$, let $T = \mathsf{N}(S)$ be the neighbors of $S$. Define $F$ as the bipartite graph with vertices $S \cup T$ and edges $\mathsf{E}(S, T) \cup \mathsf{E}(T, S)$. Let
\begin{align*}
a_S &:= \max \{ \sum_{e \in \mathsf{E}(S^, T)} w_F(e) : w_F \text{ is a maximum fractional matching}\} \\
b_S &:= \max \{ \sum_{e \in \mathsf{E}(T, S)} w_F(e) : w_F \text{ is a maximum fractional matching}\}. \\
\end{align*}
We define $g_H$ and $G(H, \delta)$ as follows.
\[
    g_H(x_1 , x_2 , y_1 , y_2) := \sum_{S \in \mathcal{S}_H} x_1^{\mathsf{v}^+(S)} x_2^{\mathsf{v}^-(S)} (x_1 \wedge x_2)
    ^{\mathsf{v}^{\pm}(S)} y_1^{a_S} y_2^{b_S}.
\]
\[
    G(H, \delta) := \inf_{\substack{0 \le x_1, x_2 \\ 0 \le y_1 , y_2 \le 1}} \{ x_1 y_1 + x_2 y_2 : g_H(x_1, x_2, y_1, y_2) = 1 + \delta \}.
    \]
\end{definition}

For the definition of a maximum fractional matching, see Section~\ref{sec:lower-bd}.

\begin{theorem}[Lower bound of variational problem]\label{thm:lower-bd}
    Let $H$ be a connected digraph with maximum degree $\Delta \ge 2$ and $\delta > 0$. Then,
    \[
        \frac{\Phi_p(H, \delta)}{p^\Delta \log(1/p)} \ge 
        \begin{cases}
            G(H, \delta) + o(1) & \text{if $H$ is irregular} \\
            \min(G(H, \delta), \delta^{2/\mathsf{v}(H)}) + o(1) & \text{if $H$ is regular}
        \end{cases}.
    \]
    Furthermore, the infimum $G(H, \delta)$ of $g_H$ is obtained when $y_1 \vee y_2 = 1$.
\end{theorem}

We remark that $a_S \le A_S$ and $b_S \le B_S$ for any $S \in \mathcal{S}_H$. This implies $f_H \le g_H$ and hence $F(H, \delta) \ge G(H, \delta)$. In other words, the upper bound is indeed greater than the lower bound. The proof to Theorem~\ref{thm:lower-bd} is given in Section \ref{sec:lower-bd}.

Finally, in Section~\ref{sec:gap}, we compare the upper and lower bounds, and give conditions on when they equal each other. Some of our results are summarized as follows.

\begin{proposition}\label{prop:sec-5-results}
    \quad
    \begin{itemize}
        \item (Proposition~\ref{prop:x1=x2}) $G(H, \delta)$ is obtained when $x_1 = x_2$ 
        \item (Proposition~\ref{prop:F-2G}) $G(H, \delta) \le F(H, \delta) \le 2G(H, \delta)$ 
        \item (Proposition~\ref{prop:2s}) If $2a_S, 2b_S \ge |S|$ for all $S \in \mathcal{S}_H$, then $G(H, \delta)$ is obtained when $y_1 = 1$. Similarly, if $2b_S \ge |S|$ for all $S \in \mathcal{S}_H$, then $G(H, \delta)$ is obtained when $y_2 = 1$.
        \item (Remark~\ref{rmk}) If $G(H, \delta)$ is obtained when $y_1 = y_2 = 1$, then $G(H, \delta) = F(H, \delta)$. 
        \item (Remark~\ref{rmk}) If $a_S = A_S$ for all $S \in \mathcal{S}$ and $G(H, \delta)$ is obtained when $y_2 = 1$, then $F(H, \delta) = G(H, \delta)$. Similar for $b_S = B_S$.
    \end{itemize}
\end{proposition}

In particular, Proposition~\ref{prop:x1=x2} implies that finding $G(H,\delta)$ is actually a $2$-dimensional variational problem. Section~\ref{sec:gap} provides also several explicit families of digraphs with $F(H, \delta) = G(H, \delta)$. In other words, the upper tail probability is reduced to a variational problem of degree $2$ (i.e., finding the minimizer of $f_H = g_H$). Some examples of such digraphs include triangles, directed cycles, stars, and balanced digraphs. It concludes with an example of a digraph with a gap between the upper and lower bounds.

\subsection*{Acknowledgements}
I thank Amir Dembo for introducing me to this research problem and his numerous helpful discussions. 

\section{Preliminaries}\label{sec:prelim}

Let $H = (\mathsf{V}, \mathsf{E})$ be a digraph. We write $\mathsf{v}(H) := |\mathsf{V}(H)|$ and $\mathsf{e}(H) = |\mathsf{E}(H)|$. For a vertex $v$ and edge $e$, $v \sim e$ or $e \sim v$ is used to denote that $v$ is an endpoint of $e$ (in either direction). For two sets $S, T \subseteq \mathsf{V}(H)$, $\mathsf{E}(S, T)$ denotes the set of edges from $S$ to $T$, and $\mathsf{e}(S, T) := | \mathsf{E}(S, T) |$. For any sets $S \subseteq \mathsf{V}(H)$, define
\begin{align*}
    \mathsf{N}^+(S) &:= \{ v \in \mathsf{V}(H) \setminus S : \exists u \in S \text{ such that } (u, v) \in \mathsf{E}(H) \} \\
    \mathsf{N}^+(S) &:= \{ v \in \mathsf{V}(H) \setminus S : \exists u \in S \text{ such that } (v, u) \in \mathsf{E}(H) \} \\
    \mathsf{N}(S) &:= \mathsf{N}^+(S) \cup \mathsf{N}^-(S) \\
    \mathsf{N}^{\pm}(S) &:= \mathsf{N}^+(S) \cap \mathsf{N}^-(S) \\
    \mathsf{N}^{+, 0}(S) &:= \mathsf{N}^+(S) \setminus \mathsf{N}^-(S) \\
    \mathsf{N}^{-, 0}(S) &:= \mathsf{N}^-(S) \setminus \mathsf{N}^+(S)
\end{align*}
Furthermore, define $\mathsf{n}(S) = \mathsf{N}(S)$ and likewise for all other sets. We remark that $A_S = \mathsf{n}^+(S) = \mathsf{n}^{+, 0} (S) + \mathsf{n}^{\pm} (S)$ and $B_S = \mathsf{n}^-(S) = \mathsf{n}^{-, 0} (S) + \mathsf{n}^{\pm} (S)$.

Now define
\begin{align*}
    \mathsf{V}^+(H) &:= \{ v \in \mathsf{V}(H) : \mathsf{n}^-(v) = 0 \} \\
    \mathsf{V}^+(H) &:= \{ v \in \mathsf{V}(H) : \mathsf{n}^+(v) = 0 \} \\
    \mathsf{V}^{\pm}(H) &:= \{ v \in \mathsf{V}(H) : \mathsf{n}^+(v), \mathsf{n}^-(v) > 0 \} \\
\end{align*}
We may sometimes abuse notation in the form $\mathsf{V}(S)$, where $S \subseteq \mathsf{V}(H)$. Like before, let $\mathsf{v}^+(H) = |\mathsf{V}^+(H)|$ and likewise for all other sets.

The \emph{homomorphism density} of $H$ in a weighted digraph $Q$ is defined to be 
\[
	t(H, Q) := \frac{1}{n^{\mathsf{v}(H)}} \sum_{\phi : \mathsf{V}(H) \to [n]} \prod_{e \in \mathsf{E}(H)} Q(\phi(e)).
\]
This can also be seen as $t(H,Q) = n^{- \mathsf{v}(H)} \hom(H, Q)$, where $\hom(H, Q)$ is the (generalized) number of homomorphisms from $H$ to $G$, i.e. $\hom(H, Q) := \sum_{\phi : \mathsf{V}(H) \to [n]} \prod_{e \in \mathsf{E}(H)} Q(\phi(e))$. We additionally denote the normalized quantities 
\[
t(H, Q) := \frac{\hom(H, Q)}{n^{\mathsf{v}(H)}}, \quad t_p (H, Q) := t(H, Q/p) = \frac{\hom(H, Q)}{n^{\mathsf{v}(H)}p^{\mathsf{e}(H)}}.
\]
We also often abuse the notation to include $t(H, G) := t(H, A_G)$, where $A_G$ is the adjacency matrix of $G$. In this case, $t(H, G)$ is the probability that the uniform random mapping of the vertices sends edges of $H$ to edges of $G$.

Let $\mathcal{A}_n := \{ Z \in \mathcal{Z}_n : Z_{i,j} \in \{0, 1\}, Z_{i,i} = 0 \}$. These elements are naturally associated to digraphs, i.e. they are the adjacency matrices of simple digraphs. Let $\mathcal{Q}_n := hull(\mathcal{A}_n)$, i.e. the set of $n \times n$ matrices with zero diagonals and all entries lying in $[0, 1]$.

The distributions $\mu_Q$ are probability measures on $\mathcal{A}_n$. Unless otherwise stated, $\mathbb{P}$ is the measure under which $\mathbf{A}$ has distribution $\mu_p$, so that $\mathbf{A}$ is the adjacency matrix for the directed Erdos-R\'enyi graph, and $\mathbb{E}$ is the associated expectation. For $Q \in \mathcal{Q}_n$, we write $\mathbb{P}_Q$, $\mathbb{E}_Q$ for probability and expectation under which $\mathbf{A}$ has the distribution $\mu_Q$. The relative entropy between the Bernoulli measures $Bernoulli(p)$ and $Bernoulli(x)$ is denoted$$ I_p(x) = D(\mu_x \| \mu_p) = x \log(x/p) + (1-x) \log((1-x)/(1-p)).$$
With abuse of notation, we generalize this to the relative entropy of $\mu_Q$ with respect to $\mu_p$, i.e.
\[
I_p(Q) = \sum_{i \ne j} I_p(Q(i,j)).
\]

$\mathbf{G}$ denotes the $p$-uniform Erdos-R\'enyi digraph. In this case,
\[
	\mathrm{UT}_{n, p} (H, \delta) := - \log \mathbb{P}(t(H, \mathbf{G}) \ge (1 + \delta) p^{\mathsf{e}(H)}) = -\log \mathbb{P}(t_p (H, G) \ge 1 + \delta)
\]
\[
	\Phi_{n, p} (H, \delta) := \inf_{Q \in \mathcal{Q}_n} \{ I_p (Q) : t(H, Q) \ge (1 + \delta) p^{\mathsf{e}(H)} \}
\]

We use the following asymptotic notation. Let $f$ and $g$ be nonnegative-valued functions of $n$. $f \ll g$ and $f = o(g)$ means that $f/g \xrightarrow{n \to \infty} 0$, while $f \lesssim g$, $f = O(g)$ means $f/g$ is bounded. $f \asymp g$ means $f \lesssim g \lesssim f$, and $f \sim g$ implies $f/g \xrightarrow{n \to \infty} 1$.

\section{Proof of Theorem~\ref{thm:upper-bd}}\label{sec:upper-bd}

In this section, we prove Theorem~\ref{thm:upper-bd}, which provides upper bounds for the variational problem
\[
	\Phi_{n, p} (H, \delta) := \inf_{Q \in \mathcal{Q}_n} \{ I_p (Q) : t(H, Q) \ge (1 + \delta) p^{\mathsf{e}(H)} \}.
\]
These bounds are given by calculating $I_p(Q)$ for specific choices of $Q$ with $t(H, Q) \ge (1 + \delta) p^{\mathsf{e}(H)}$. In particular, we consider two possibilities: planting a small clique or hub. The clique case is nearly identical to the undirected case. However, when planting a hub, we need to consider \emph{directed} hubs in order to compensate for the asymmetry in digraphs. To this end, we define the following function $\bar{f}_H$.

\begin{align*}
    \bar{f}_H(x_1 , x_2 , y_1 , y_2) &:= \sum_{S \in \mathcal{S}_H} x_1^{\mathsf{v}^+ (S)}
    x_2^{\mathsf{v}^- (S)} 
    (x_1 \wedge x_2)^{\mathsf{v}^{\pm} (S)}
    y_1^{\mathsf{n}^{+, 0} (S)}
    y_2^{\mathsf{n}^{-, 0} (S)}
    (y_1 \wedge y_2)^{\mathsf{n}^{\pm} (S)}
\end{align*}

\begin{proposition}\label{prop:upper-bd}
    \quad
    \begin{enumerate}[(a)]
        \item Let $H$ be a connected, $\Delta$-regular digraph and $\delta > 0$. If $n^{-2/\Delta} \ll p \ll 1$, then
        \[
        \frac{\Phi_{n, p}(H, \delta)}{n^2 p^\Delta \log (1/p)} \le (\delta^{2 / \mathsf{v}(H)} + o(1)).
        \]
        \item Let $H$ be a digraph with maximum degree $\Delta$ and $\delta > 0$. If $n^{-1/\Delta} \ll p \ll 1$, then
        \[
        \frac{\Phi_{n, p}(H, \delta)}{n^2 p^\Delta \log(1/p)} \le \inf_{\substack{0 \le x_1 , x_2 \\ 0 \le y_1 , y_2 \le 1}} \{ x_1 y_1 + x_2 y_2 : \bar{f}_H(x_1 , x_2 , y_1 , y_2) = 1 + \delta\} + o(1).
        \]
        Further, the infimum of the right-hand side is obtained when $y_1 \vee y_2 = 1$.
    \end{enumerate}
\end{proposition}

\begin{proof}

The first claim almost identical to the undirected case proven in \cite{BGLZ2017},  except for the fact that $I_p(Q) \sim \delta^{2/\mathsf{v}(H)} n^2 p^\Delta \log(1/p)$ instead of $\frac{1}{2}\delta^{2/\mathsf{v}(H)} n^2 p^\Delta \log(1/p)$ due the edges existing in both directions.

Now we prove part (b). We first show that the infimum is obtained when $y_1 \vee y_2 = 1$. To see this, note that the sum of degrees for the $y_i$ terms in $\bar{f}_H$ are always at least $|S|$. As such, linearly increasing $y_1$ and $y_2$ while keeping $x_1 y_1 + x_2 y_2$ constant will always increase the value of $\bar{f}_H$. Therefore, it is always optimal to choose $y_1 \vee y_2 = 1$. 

Now assume $y_1 \vee y_2 = 1$ and let $x_3 = x_1 \wedge x_2$, $y_3 = y_1 \wedge y_2$. Choose $A_1 , A_2 \subseteq [n]$ such that $|A_1| \sim x_1 p^{\Delta} n$, $|A_2| \sim x_2 p^{\Delta} n$, and $|A_1 \cap A_2| \sim x_3  p^\Delta n$. Further, choose $B_1 , B_2$ containing $A_1 \cup A_2$ such that $|B_1| \sim y_1 n$, $|B_2| \sim y_2 n$, and $|B_1 \cap B_2| \sim y_3 n$. Define $A_3 = A_1 \cap A_2$, $B_3 = B_1 \cap B_2$. Now consider a matrix $Q \in \mathcal{Q}_n$ such that
\[
Q(i, j) = 
\begin{cases}
1 & (i \in A_1 \text{ and } j \in B_1) \text{ or } (j \in A_2 \text{ and } i \in B_2) \\
p & \text{otherwise}.
\end{cases}
\]
For a visualization of this digraph, see Figure~\ref{fig:graphon}. This gives $I_p(Q) \sim (|A_1||B_1| + |A_2||B_2|) I_p(1) \sim (x_1 y_1 + x_2 y_2) n^2 p^{\Delta} \log(1/p)$. Hence, it only remains to show $t(H, Q) \sim (1 + \delta) p^{\mathsf{e}(H)}$, which is done in the following argument.

\begin{align*}
t(H, Q) &\sim \sum_{\substack{X_1 , X_2 , X_3 , Y_1 , Y_2 , Y_3 \\ \text{: partition of } \mathsf{V}(H)}} \left( \frac{|A_1 \setminus A_3|}{n} \right)^{|X_1|}\left( \frac{|A_2 \setminus A_3|}{n} \right)^{|X_2|}\left( \frac{|A_3|}{n} \right)^{|X_3|} \\
& \left( \frac{|B_1 \setminus B_3|}{n} \right)^{|Y_1|}\left( \frac{|B_1 \setminus B_3|}{n} \right)^{|Y_2|}\left( \frac{|B_3|}{n} \right)^{|Y_3|} p^{\mathsf{e}(H) - \mathsf{e} (X_1 \cup X_3 , Y_1 \cup Y_3) - \mathsf{e}^(Y_2 \cup Y_3 , X_2 \cup X_3)} \\
&\sim \sum_{\substack{X_1 , X_2 , X_3 , Y_1 , Y_2 , Y_3 \\ \text{: partition of } \mathsf{V}(H)}} (x_1 - x_3)^{|X_1|} (x_2 - x_3)^{|X_2|} x_3^{|X_3|} (y_1 - y_3)^{|Y_1|} (y_2 - y_3)^{|Y_2|} y_3^{|Y_3|} \\
&p^{\mathsf{e}(H)} p^{\Delta(|X_1| + |X_2| + |X_3|) - \mathsf{e} (X_1 \cup X_3 , Y_1 \cup Y_3) - \mathsf{e}^(Y_2 \cup Y_3 , X_2 \cup X_3)} \\
&\sim \sum_{ \tiny \substack{S \in \mathcal{S}_H \\ X_1 \subseteq \mathsf{V}^{+}(S) \\ X_2 \subseteq \mathsf{V}^{-}(S) \\ X_3 = S \setminus(X_1 \cup X_2) \\ Y_1 \subseteq \mathsf{N}^{+, 0}(S) \\ Y_2 \subseteq \mathsf{N}^{-, 0}(S) \\ Y_3 = (\mathsf{V}(H) \setminus S) \setminus (Y_1 \cup Y_2) }} (x_1 - x_3)^{|X_1|} (x_2 - x_3)^{|X_2|} x_3^{|X_3|} (y_1 - y_3)^{|Y_1|} (y_2 - y_3)^{|Y_2|} y_3^{|Y_3|} p^{\mathsf{e}(H)} \\
&= \sum_{ \tiny \substack{S \in \mathcal{S}_H \\ X_1 \subseteq \mathsf{V}^{+}(S) \\ X_2 \subseteq \mathsf{V}^{-}(S) \\ Y_1 \subseteq \mathsf{N}^{+, 0}(S) \\ Y_2 \subseteq \mathsf{N}^{-, 0}(S)}} (x_1 - x_3)^{|X_1|} (x_2 - x_3)^{|X_2|} x_3^{|S| - |X_1| - |X_2|}  (y_1 - y_3)^{|Y_1|} (y_2 - y_3)^{|Y_2|} y_3^{|\mathsf{V}(H) \setminus S| - |Y_1| - |Y_2|} p^{\mathsf{e}(H)} \\
&= \sum_{S \in \mathcal{S}_H} x_1^{\mathsf{v}^+ (S)}
x_2^{\mathsf{v}^- (S)} 
x_3^{\mathsf{v}^{\pm} (S)}
y_1^{\mathsf{n}^{+, 0} (S)}
y_2^{\mathsf{n}^{-, 0} (S)}
y_3^{\mathsf{n}^{\pm} (S)} p^{\mathsf{e}(H)} \\
&= \bar{f}_H(x_1 , x_2 , y_1 , y_2) p^{\mathsf{e}(H)} \\
&= (1 + \delta) p^{\mathsf{e}(H)}
\end{align*}

The third approximation is due to the fact that $\Delta(|X_1| + |X_2| + |X_3|) \ge \mathsf{e}(X_1 \cup X_3, Y_1 \cup Y_3) + \mathsf{e}(Y_2 + Y_3 , X_2 \cup X_3)$ with equality only under the conditions of the third summation. Indeed, the left-hand side is at least  the sum of degrees of $X_1 \cup X_2 \cup X_3$, while the right-hand side is at most the sum of out-degrees of $X_1 \cup X_3$ and the in-degrees of $X_2 \cup X_3$. Hence it is clear that the equality only holds when $X_1 \cap X_2 \cap X_3$ is an independent set of $H^*$, $X_1$ and $Y_2$ have no in-degrees, and $X_2$ and $Y_1$ have no out-degrees. The fifth line is a consequence of the binomial theorem along with the fact that $|S| = \mathsf{v}^+(S) + \mathsf{v}^+(S) + \mathsf{v}^{\pm}(S)$ and similarly for $\mathsf{V}(H) \setminus S$.

\end{proof}

\begin{proof}[Proof of Theorem~\ref{thm:upper-bd}]
    
Since $A_S + B_S \ge |\mathsf{N}(S)| \ge |S|$, we can see that the infimum is obtained when $y_1 \vee y_2 = 1$ for similar reasons as above. Now note that
\[
    \bar{f}_H(x_1 , x_2 , y_1, 1) = \sum_{S \in \mathcal{S}_H} x_1^{\mathsf{v}^+ (S)}
    x_2^{\mathsf{v}^- (S)} 
    (x_1 \wedge x_2)^{\mathsf{v}^{\pm} (S)}
    y_1^{\mathsf{n}^{+, 0} (S) + \mathsf{n}^{\pm} (S)} = f_H(x_1 , x_2 , y_1 , 1)
\]
and similarly for $\bar{f}_H(x_1 , x_2 , 1, y_2)$. Therefore, Proposition~\ref{prop:upper-bd} implies Theorem~\ref{thm:upper-bd}.
    
\end{proof}

\section{Proof of Theorem~\ref{thm:lower-bd}}\label{sec:lower-bd}

In this section, we assume $H$ is a connected digraph with maximum degree $\Delta \ge 2$.
\subsection{Graphon variational problem}

In order to show the lower bound, we use a continuous analog of digraphs called \emph{directed graphons}. A directed graphon is defined to be a measurable function $W : [0, 1]^2 \to [0, 1]$. Unlike regular graphons, we do not require symmetry. This definition has a natural connection with the matrices $\mathcal{Q}_n$, which can be interpreted as discrete approximations of directed graphons. We write $\mathbb{E} [f(W)] = \int_{[0, 1]^2} f(W(x, y)) dx dy$. From now on, we will discuss the following continuous version of the variational problem.

For $\delta > 0$ and $0 < p < 1$, let
$$
\Phi_p(H, \delta) := \inf \left\{ \mathbb{E}[I_p(W)] : W \text{ is a directed graphon with } t(H, W) \ge (1 + \delta)p^{\mathsf{e}(H)} \right\}
$$
where
$$
t(H, W) := \int_{[0,1]^{\mathsf{v}(H)}} \prod_{(i, j) \in \mathsf{E}(H)} W(x_i, x_j) dx_1 , \dots , dx_{\mathsf{v}(H)}.
$$
The following lemma shows that finding a lower bound for $\Phi_p(H, \delta)$ gives a lower bound for $\Phi_{n, p}(H, \delta)$.

\begin{lemma}
    For any $H, p, n, \delta$, we have $\Phi_{p}(H, \delta) \le n^{-2} \Phi_{n, p}(H, \delta)$.
\end{lemma}

\begin{proof}
    The proof is straightforward and also done in \cite{BGLZ2017}.
\end{proof}

The graphon analogue of the solution candidates for given in the previous section are depicted in Figure~\ref{fig:graphon}.

\tikzset{every picture/.style={line width=0.75pt}} 

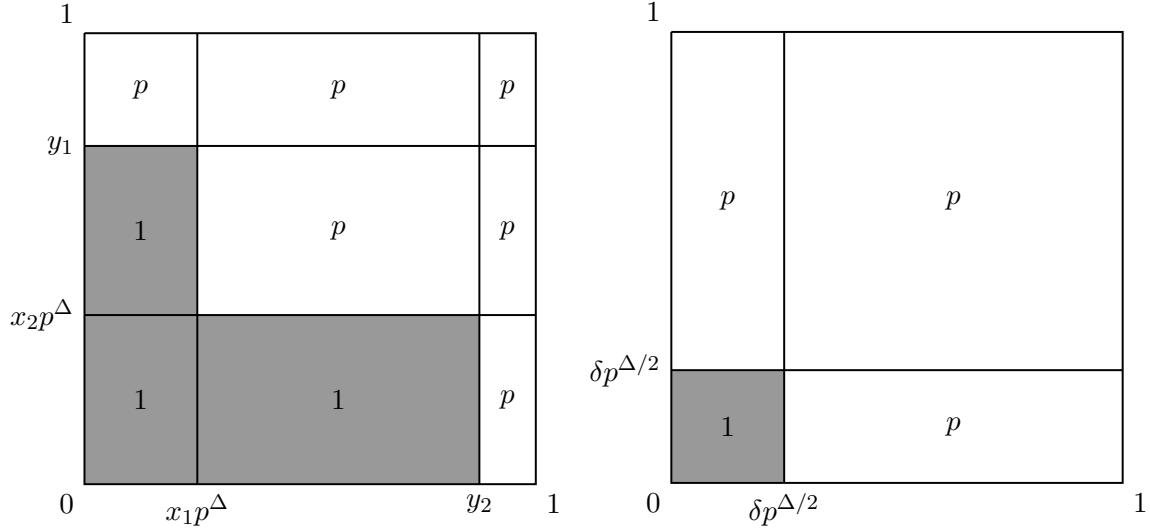
\begin{figure}
    \centering
    \label{fig:graphon}
    \begin{tikzpicture}[xscale=1.5,yscale=1.5]
        \fill[gray!80] (0,0) -- (0,3) -- (1,3) -- (1,0);
        \fill[gray!80] (0,0) -- (0,1.5) -- (3.5,1.5) -- (3.5,0);
        
        \draw (0,4) -- (4,4) -- (4,0) node[anchor=north west]{$1$}-- (0,0) node[anchor=north east]{$0$} -- (0,4) node[anchor=south east]{$1$};
        \draw (1,4) -- (1,0) node[anchor=north]{$x_1 p^\Delta$};
        \draw (4,1.5) -- (0,1.5) node[anchor=east]{$x_2 p^\Delta$};
        \draw (3.5,4) -- (3.5,0) node[anchor=north]{$y_2$};
        \draw (4,3) -- (0,3) node[anchor=east]{$y_1$};

        \node at (0.5,0.75) {$1$};
        \node at (2.25,0.75) {$1$};
        \node at (0.5,2.25) {$1$};
        \node at (2.25,2.25) {$p$};
        \node at (3.75,2.25) {$p$};
        \node at (2.25,3.5) {$p$};
        \node at (3.75,3.5) {$p$};
        \node at (3.75,0.75) {$p$};
        \node at (0.5,3.5) {$p$};
        \end{tikzpicture}
        \begin{tikzpicture}[xscale=1.5,yscale=1.5]
        
        \fill[gray!80] (0,0) -- (0,1) -- (1,1) -- (1,0);

        \draw (0,4) -- (4,4) -- (4,0) node[anchor=north west]{$1$}-- (0,0) node[anchor=north east]{$0$} -- (0,4) node[anchor=south east]{$1$};
        \draw (1,4) -- (1,0) node[anchor=north]{$\delta p^{\Delta/2}$};
        \draw (4,1) -- (0,1) node[anchor=east]{$\delta p^{\Delta/2}$};

        \node at (0.5,0.5) {$1$};
        \node at (0.5,2.5) {$p$};
        \node at (2.5,0.5) {$p$};
        \node at (2.5,2.5) {$p$};
        \end{tikzpicture}
\caption{Candidates for the graphon variational problem}
\end{figure}

\subsection{Decomposition of $t(H, W)$}

Suppose $W$ is a directed graphon satisfying $t(H, W) \ge (1 + \delta) p^{\mathsf{e}(H)}$. Often, we will use the shorthand
$$
W(\boldsymbol{x} | H) := \prod_{(i, j) \in \mathsf{E}(H)} W(x_i, x_j)
$$
where $\boldsymbol{x} = (x_v)_{v \in \mathsf{V}(H)}$. Now note that $I_p$ has minimum value at $p$, so we may assume $W \ge p$. Letting $U := W - p$, this implies $0 \le U \le 1-p$. We can expand $t(H, W) = t(H, p + U)$ as
$$
t(H, W) - p^{\mathsf{e}(H)} = \sum_{F} t(F, U) p^{\mathsf{e}(H) - \mathsf{e}(F)},
$$
where the sum is taken over all non-empty subgraphs $F$ of $H$. The following lemma shows that the only non-negligible contributions occur when $F$ is $\Delta$-regular or $F \in \mathcal{F}_H$, where
$$
\mathcal{F}_H := \{ F : F \text{ is a non-empty subgraph of $H$ with } \tau(F) = \mathsf{e}(F)/\Delta \}.
$$
Here, $\tau(F)$ is the minimum size of a vertex cover of $F$.

\begin{lemma}[{\cite[Corollary~6.2]{BGLZ2017}}]
    Let $H$ be a connected digraph with maximum degree $\Delta \ge 2$, and $F$ a non-empty subgraph of $H$. If $U$ is a directed graphon satisfying $\mathbb{E}[I_p(p + U)] \lesssim p^{\Delta} I_p(1)$, then $t(F, U) = o(p^{\mathsf{e}(H)})$ unless $F \in \mathcal{F}_H$ or $F$ is $\Delta$-regular (i.e. $F = H$ and $H$ is $\Delta$-regular). In other words,
$$
t(H, W) - p^{\mathsf{e}(H)} = \sum_{F \in \mathcal{F}_H \cup \{ H \}} t(F, U) p^{\mathsf{e}(H) - \mathsf{e}(F)} + o(p^{\mathsf{e}(H)}).
$$
If $H$ is irregular, the summation may be taken over just $\mathcal{F}_H$.
\end{lemma}

\begin{proof}
    The proof is identical to the case of undirected graphs \cite[Corollary~6.2]{BGLZ2017}.
\end{proof}

The following are some basic properties of $\mathcal{F}_H$ listed in \cite{BGLZ2017}.

\begin{itemize}
    \item Every $F \in \mathcal{F}_H$ is bipartite and has maximum degree exactly $\Delta$.
    \item If $S$ is a minimum size vertex cover of $F$, then $S$ is an independent set of vertices, all of which have degree $\Delta$ in $F$. Further, $(S, \mathsf{V}(F) \setminus S)$ forms a vertex bipartition of $F$.
    \item Conversely, any non-empty independent set $S$ of $H^*$ uniquely determines an $F \in \mathcal{F}_H$, where $F$ is generated by the edges of $H$ incident to $S$.
    \item If $H$ is regular and bipartite, then the two parts of the vertex bipartition are the only minimum vertex covers of $H \in \mathcal{F}_H$. For all other $F \in \mathcal{F}_H$, there is only one minimum vertex cover.
    \item For a regular digraph $H$, $H \in \mathcal{F}_H$ if and only if $H$ is bipartite.
\end{itemize}

\subsection{Contributions from non-negligible terms}

We now proceed to  calculate the contributions given by non-negligible terms. To this end, we define the sets of points with high in/out degrees:
\begin{align*}
B_b^+ = B_b^+(U) := \{ x : d_U^+ (x) \ge b \}, \quad \text{where} \quad d^+(x) = d_U^+(x) := \int_0^1 U(x, y) dy \\
B_b^- = B_b^-(U) := \{ x : d_U^- (x) \ge b \}, \quad \text{where} \quad d^-(x) = d_U^-(x) := \int_0^1 U(y, x) dy.
\end{align*}

Further, let
$$
B_b := B_b^+ \cup B_b^-, \quad B_b^{\pm} := B_b^+ \cap B_b^-.
$$

\begin{lemma}\label{lem:contri}
    Let $\Delta \ge 2$ and $U$ be a directed graphon satisfying $\mathbb{E}[I_p(p + U)] \lesssim p^{\Delta} I_p(1)$.

    \begin{enumerate}[(a)]
        \item Let $F$ be a connected irregular bipartite digraph with maximum degree $\Delta$ and $\tau(F) = \mathsf{e}(F) / \Delta$. Let $S$ be the unique vertex cover of $F$ with size $\mathsf{e}(F) / \Delta$ and $T = \mathsf{V}(F) \setminus S$ so that $(S, T)$ forms a vertex bipartition of $F$. Then, for any $p^{1/3} \ll b \ll 1$,
        \begin{align*}
            t(F, U) &= \int U(\boldsymbol{x} | F) \boldsymbol{1}
            \left\{
            \begin{array}{ll}
            \forall v \in \mathsf{V}^+(S) \colon x_v \in B_b^+ \\
            \forall v \in \mathsf{V}^-(S) \colon x_v \in B_b^- \\
            \forall v \in \mathsf{V}^{\pm}(S) \colon x_v \in B_b^{\pm} \\
            \forall u \in T \colon x_u \in \overline B_b
            \end{array}
            \right\}
            d\boldsymbol{x} + o(p^{\mathsf{e}(F)}) \\
        \end{align*}
        \item Let $H$ be a connected $\Delta$-regular non-bipartite digraph. Then there exists some $\kappa = \kappa(H) > 0$ such that for any $p^{\kappa} \le b \ll 1$,
        $$
        t(H, U) = \int U(\boldsymbol{x}|H) \boldsymbol{1} \{ \forall v \in \mathsf{V}(H): x_v \in \bar{B}_b \} d \boldsymbol{x} + o(p^{\mathsf{e}(F)}).
        $$
        \item Let $H$ be a connected $\Delta$-regular bipartite digraph with vertex bipartition $(S, T)$. For any $b_0 = o(1)$, there exists some $b$ with $b_0 \le b \ll 1$ such that
        $$
        t(H, U) = \Gamma_1 + \Gamma_2 + \Gamma_3 + o(p^{\mathsf{e}(H)}),
        $$
        where
        \begin{align*}
        \Gamma_1 &= \int U(\boldsymbol{x} | H) \boldsymbol{1}
         \left\{
         \begin{array}{ll}
         \forall v \in \mathsf{V}^+(S) \colon x_v \in B_b^+ \\
         \forall v \in \mathsf{V}^-(S) \colon x_v \in B_b^- \\
         \forall v \in \mathsf{V}^{\pm}(S) \colon x_v \in B_b^{\pm} \\
         \forall u \in T \colon x_u \in \overline B_b
         \end{array}
         \right\}
         d\boldsymbol{x} \\
         \Gamma_2 &= \int U(\boldsymbol{x} | H) \boldsymbol{1}
         \left\{
         \begin{array}{ll}
         \forall v \in \mathsf{V}^+(T) \colon x_v \in B_b^+ \\
         \forall v \in \mathsf{V}^-(T) \colon x_v \in B_b^- \\
         \forall v \in \mathsf{V}^{\pm}(T) \colon x_v \in B_b^{\pm} \\
         \forall u \in S \colon x_u \in \overline B_b
         \end{array}
         \right\}
         d\boldsymbol{x} \\
         \Gamma_3 &= \int U(\boldsymbol{x}|H) \boldsymbol{1} \{ \forall v \in \mathsf{V}(H): x_v \in \bar{B}_b \} d \boldsymbol{x}.
         \end{align*}
        $\Gamma_1, \Gamma_2$ can also be decomposed further into sums of the form similar to (a).
    \end{enumerate}
\end{lemma}

\begin{proof}
    The proof is almost identical to the proof of \cite[Proposition~6.5]{BGLZ2017}, with one minor modification. That is, we need a slightly stronger version of \cite[Lemma~7.1]{BGLZ2017}. This is stated the next lemma. However, the proof of \cite[Lemma~7.1]{BGLZ2017} already suffices to prove this stronger claim, and so we will omit the proof here.
\end{proof}

\begin{lemma}
    Let $F$ be a connected irregular bipartite digraph with maximum degree $\Delta$ and $\tau(F) = \mathsf{e}(F) / \Delta$. For every edge $e$ of $F$, there is a $2$ matching $M$ of $F$ of size $2\mathsf{e}(F) / \Delta$ containing $e$ such that the connected component of $e$ in $M$ is a path.
\end{lemma}

\begin{remark}
    Many parts of the proofs in \cite{BGLZ2017} rely on combinatorial properties of graphs $F$ and $H$. We remark that none of these properties rely on the fact that these graphs are simple or undirected, and so they can be applied to our setting without alteration.
\end{remark}

\subsection{Upper bound on contributions}

For any $\epsilon > 0$, let $\Gamma_p^\epsilon (H, \delta)$ be the set of directed graphons satisfying
\begin{enumerate}[(a)]
    \item $t(H, W) \ge (1 + \delta) p^{\mathsf{e}(H)}$
    \item $W$ only takes values in $\{ p \} \cup [(1 + \epsilon)p, 1]$.
\end{enumerate}

Define
\[
    \Phi_p^\epsilon(H, \delta) := \inf \left\{ \mathbb{E}[I_p (W)] : W \in \Gamma_p^\epsilon(H, \delta) \right\}.
\]

\begin{lemma}[{\cite[Lemma~6.1]{BGLZ2017}}]
    \label{lem:phi-ep}
    \[
        \Phi_{p}^\epsilon(H, \delta) \le \Phi_{p}(H, (1 + \epsilon)^{\mathsf{e}(H)}(1 + \delta) - 1).
    \]
\end{lemma}

The upshot of this lemma is that we may disregard directed graphons with values near $p$, other than $p$ itself. This technique was also used in \cite{LZ2021} and \cite{G2021}. Thus, we can use the following lemma to get good bounds on $\mathbb{E}[I_p(W)]$.

\begin{lemma}[{\cite[Lemma~4.3]{G2021}}]
    \label{lem:near-p}
    For all $c > 0$, there exists $\epsilon > 0$ such that if $p \to 0$ and $x = x(p)$ satisfies $|x| \ge p^{1 + \epsilon}$, then
    \[
        I_p(p + x) \ge (1 - o(1)) |x|^{1 + c} \log (1/p).
    \]
\end{lemma}

From now on, fix $\epsilon$ such that $\epsilon = o(1)$ and $\epsilon = p^{o(1)}$ as $p \to 0$ ($\epsilon = 1/\log(1/p)$ is one example). By Lemma~\ref{lem:phi-ep}, we can choose $W = W(p) \in \Gamma_p^{\epsilon}(H, \delta)$ such that $\mathbb{E}[I_p(W)] \le (1 + o(1)) \Phi_p (H, \delta)$. Let $U = W - p$ and define

\begin{align*}
\theta_b^+ &:= \left( p^\Delta \log \frac{1}{p} \right)^{-1} \int_{B_b^+ \times \bar{B_b}} I_p(p + U(x, y)) dx dy \\
\theta_b^- &:= \left( p^\Delta \log \frac{1}{p} \right)^{-1} \int_{\bar{B_b} \times B_b^-} I_p(p + U(x, y)) dx dy \\
\theta_b^{\pm, +} &:= \left( p^\Delta \log \frac{1}{p} \right)^{-1} \int_{B_b^{\pm} \times \bar{B_b}} I_p(p + U(x, y)) dx dy \\
\theta_b^{\pm, -} &:= \left( p^\Delta \log \frac{1}{p} \right)^{-1} \int_{\bar{B_b} \times B_b^{\pm}} I_p(p + U(x, y)) dx dy \\
\eta_b &:= p^{-\Delta} \int_{\bar{B_b} \times \bar{B_b}} I_p(p + U(x, y)) dx dy.
\end{align*}

Note that $\theta_b^+ \ge \theta_b^{\pm, +}$ and $\theta_b^- \ge \theta_b^{\pm, -}$. We proceed to give an upper bound on the values $t(F, U)$ in terms of the above values. This is an application of Finner's inequality.

\begin{lemma}[Finner's inequality, \cite{F1992}]
    \label{lem:finner}
    Let $\mu_j$ be a probability measure on $\Omega_j$ for $j \in [n] = \{ 1 , \dots , n \}$and let $\mu = \prod_{j=1}^n \mu_j$. For nonempty subsets $A_1 , \dots , A_m$ of $[n]$, let $\mu_A = \prod_{j \in A} \mu_j$ and $\Omega_A = \prod_{j \in A} \Omega_j$. Le $f_i \in L^{p_i} (\Omega_{A_i}, \mu_{A_i})$ for each $i \in [m]$. If $\sum_{i : A_i \ni j} (1/p_i) \le 1$ for all $j \in [n]$, then
    \[
        \int \prod_{i=1}^m |f_i| d\mu \le \prod_{i=1}^m \left( \int |f_i|^{p_i} d\mu_{A_i} \right)^{1/p_i}.
    \]    
\end{lemma}

\begin{lemma}\label{lem:holder}
    Suppose $F \in \mathcal{F}_H$ with minimum vertex cover $S$ and vertex bipartition $(S, T)$. Let $w_F : \mathsf{E}(F) \to [0, 1]$ be a maximum fractional matching on $F$. In other words, $w_F$ satisfies the following two conditions.
    \begin{enumerate}[(a)]
        \item $\sum_{e \sim v} w_F(e) \le 1$ for every $v \in \mathsf{V}(F)$.
        \item $\sum_{e \sim v} w_F(e) = 1$ for every $v \in S$.
    \end{enumerate}
Then, 
\begin{multline*}
    \int U(\boldsymbol{x} | F) \boldsymbol{1}
    \left\{
    \begin{array}{ll}
    \forall v \in \mathsf{V}^+(S) \colon x_v \in B_b^+ \\
    \forall v \in \mathsf{V}^-(S) \colon x_v \in B_b^- \\
    \forall v \in \mathsf{V}^{\pm}(S) \colon x_v \in B_b^{\pm} \\
    \forall u \in T \colon x_u \in \overline B_b
    \end{array}
    \right\}
    d\boldsymbol{x} \\
    \le (\theta_b^+)^{\mathsf{v}^+(S)}(\theta_b^-)^{\mathsf{v}^-(S)}(\theta_b^{\pm, +})^{\sum_{\mathsf{E}(\mathsf{V}^{\pm}(S), T)} w_F(e)}(\theta_b^{\pm, -})^{\sum_{\mathsf{E}(T, \mathsf{V}^{\pm}(S))} w_F(e)} + o(p^{\mathsf{e}(F)}).
\end{multline*}
\end{lemma}

\begin{proof}
    First assume $w_F (e) < 1$ for all $e \in \mathsf{E}(F)$. $w_F \equiv 1/\Delta$ is an example of such a function. The proof is a direct application of Finner's inequality to the result of Lemma~\ref{lem:contri} with weights $1/w_F(e)$. Indeed, each $e = (v, u) \in \mathsf{E}(S, F)$ contributes a factor of $\left( \int U(x, y)^{1 / w_F(e)} \right)^{w_F(e)}$ to the upper bound, where the integral is taken over $B_b^+ \times \bar{B_b}$ if $v \in \mathsf{V}^+(S)$ and $B_b^{\pm} \times \bar{B_b}$ if $v \in \mathsf{V}^{\pm}(S)$. By Lemma~\ref{lem:near-p}, this is bounded by $(p^\Delta \theta_b^+)^{w_F(e)}$ and $(p^\Delta \theta_b^{\pm, +})^{w_F(e)}$, respectively. Using a similar result for all edges $(u, v) \in \mathsf{E}(F)$, we can deduce
    \begin{align*}
    \int U(\boldsymbol{x} | F) \boldsymbol{1}
        &\left\{
        \begin{array}{ll}
        \forall v \in S^+ \colon x_v \in B_b^{+} \\
        \forall v \in S^- \colon x_v \in B_b^{-} \\
        \forall v \in S^{\pm} \colon x_v \in B_b^{\pm} \\
        \forall u \in T \colon x_u \in \overline B_b
        \end{array}
        \right\}
        d\boldsymbol{x}
	\\
    &\le p^{\mathsf{e}(F)} \prod_{\substack{v \in \mathsf{V}^+(S) \\ e \sim v}} (\theta_b^+)^{w_F(e)}\prod_{\substack{v \in \mathsf{V}^-(S) \\ e \sim v}} (\theta_b^-)^{w_F(e)} \prod_{e \in \mathsf{E}(\mathsf{V}^{\pm}(S), T)} (\theta_b^{\pm, +})^{w_F(e)}\prod_{e \in \mathsf{E}(T, \mathsf{V}^{\pm}(S))} (\theta_b^{\pm, -})^{w_F(e)}  \\
	&=  p^{\mathsf{e}(F)} (\theta_b^+)^{\mathsf{v}^+(S)}(\theta_b^-)^{\mathsf{v}^-(S)}(\theta_b^{\pm, +})^{\sum_{\mathsf{E}(\mathsf{V}^{\pm}(S), T)} w_F(e)}(\theta_b^{\pm, -})^{\sum_{\mathsf{E}(T, \mathsf{V}^{\pm}(S))} w_F(e)} + o(p^{\mathsf{e}(F)}).
    \end{align*}

    Now consider the case where $w_F(e) \le 1$. Note that $(1-t)w_F + t/\Delta$ is also a maximum fractional matching whenever $t \in (0, 1)$, and that all weights are strictly less than $1$. Using the previous argument and taking $t \to 0$ sufficiently slowly, we can prove the same result even when $w_F(e) \le 1$.
\end{proof}

\subsection{Proof of Theorem~\ref{thm:lower-bd}}

From the above, we can now derive a lower bound for the variational problem. Out of all possible choices for $w_F$, it suffices to consider the two extreme cases where the exponents of $\theta_b^{\pm, +}$ are maximized or minimized. To this end, recall the following definition.

\begin{definition}
    For each $S \in \mathcal{S}_H$, let $F \in \mathcal{F}_H$ be the bipartite digraph associated with $S$ and $T = N(S)$. Recall the definition of $a_S$ and $b_S$.
\begin{align*}
a_S &:= \max \{ \sum_{e \in \mathsf{E}(S, T)} w_F(e) : w_F \text{ is a maximum fractional matching}\} \\
b_S &:= \max \{ \sum_{e \in \mathsf{E}(T, S)} w_F(e) : w_F \text{ is a maximum fractional matching}\}. \\
\end{align*}
\end{definition}

\begin{proof}[Proof of Theorem~\ref{thm:lower-bd}(a)]
    Suppose $W$ satisfies $t(H, W) \ge (1 + \delta) p^{\mathsf{e}(H)}$. Let $x_3 = \theta_b^{\pm, +} \vee \theta_b^{\pm, -}$ and choose $x_1, x_2, y_1, y_2$ such that $\theta_b^+ = x_1 y_1$, $\theta_b^- = x_2 y_2$, $\theta_b^{\pm, +} = x_3 y_1$ and $\theta_b^{\pm, -} = x_3 y_2$. It is clear from this definition that $0 \le y_1 , y_2 \le y_1 \vee y_2 = 1$ and $0 \le x_3 \le x_1 \wedge x_2$.
    
    Suppose $S \in \mathcal{S}_H$ with the associated bipartite graph $F \in \mathcal{F}_H$. Let $w_F$ be the maximum fractional matching with $\sum_{e \in \mathsf{E}(S, T)} w_F(e) = a_S$. By Lemma~\ref{lem:holder}, we get
\begin{align*}
    \int U(\boldsymbol{x} | F) \boldsymbol{1}
    &\left\{
    \begin{array}{ll}
    \forall v \in S^+ \colon x_v \in B_b^{+} \\
    \forall v \in S^- \colon x_v \in B_b^{-} \\
    \forall v \in S^{\pm} \colon x_v \in B_b^{\pm} \\
    \forall u \in T \colon x_u \in \overline B_b
    \end{array}
    \right\}
    d\boldsymbol{x}  \\
    &\le  p^{\mathsf{e}(F)} (\theta_b^+)^{\mathsf{v}^+(S)}(\theta_b^-)^{\mathsf{v}^-(S)}(\theta_b^{\pm, +})^{\sum_{\mathsf{E}(\mathsf{V}^{\pm}(S), T)} w_F(e)}(\theta_b^{\pm, -})^{\sum_{\mathsf{E}(T, \mathsf{V}^{\pm}(S))} w_F(e)} + o(p^{\mathsf{e}(F)}) \\
    &= p^{\mathsf{e}(F)} x_1^{\mathsf{v}^+(S)} x_2^{\mathsf{v}^(S)} x_3^{\mathsf{v}^{\pm}(S)} y_1^{a_S} y_2^{\mathsf{v}(S) - a_S} + o(p^{\mathsf{e}(F)}) \\
    &\le p^{\mathsf{e}(F)} x_1^{\mathsf{v}^+(S)} x_2^{\mathsf{v}^-(S)} (x_1 \wedge x_2)^{\mathsf{v}^{\pm}(S)} y_1^{a_S} + o(p^{\mathsf{e}(F)}).
\end{align*}

By multiplying both sides of this inequality by $p^{\mathsf{e}(H) - \mathsf{e}(F)}$ and summing over all $S \in \mathcal{S}_H$, we get
\[
    t(H, W) \le p^{\mathsf{e}(H)}g_H(x_1, x_2 , y_1 , 1) + o(p^{\mathsf{e}(H)}).
\]
Repeating a similar process for $b_S$, we have
\[
    t(H, W) \le p^{\mathsf{e}(H)}g_H(x_1, x_2 , 1 , y_2) + o(p^{\mathsf{e}(H)}).
\]
Since $y_1 \vee y_2 = 1$ this gives
\[
    t(H, W) \le p^{\mathsf{e}(H)}g_H(x_1, x_2 , y_1 , y_2) + o(p^{\mathsf{e}(H)}).
\]
$t(H, W) \ge (1 + \delta) p^{\mathsf{e}(S)}$, so this implies $1 + \delta \le g_H(x_1 ,x_2 , y_1 , y_2)$. Therefore,
\begin{align*}
    \mathbb{E}[I_p(W)] &\ge p^\Delta \log(1/p) (\theta_b^+ + \theta_b^-) \\ &= p^{\Delta} \log(1/p) ((x_1y_1 + x_2 y_2) + o(1)) \\
    &\ge p^{\Delta} \log (1/p) G(H, \delta).
\end{align*}
Taking the infimum over all $t(H, W) \ge (1 + \delta)p^{\mathsf{e}(H)}$ completes the proof.
\end{proof}

For part (b), we need the following two lemmas.
\begin{lemma}[{\cite[Proposition~6.5]{BGLZ2017}}]
    Suppose $H$ is $\Delta$-regular.
    \[
        \int U(\boldsymbol{x}|H) \boldsymbol{1} \{ \forall v \in \mathsf{V}(H): x_v \in \bar{B}_b \} d \boldsymbol{x} \le \eta_b^{\mathsf{e}(H) / \Delta} p^{\mathsf{e}(H)} + o(p^{\mathsf{e}(H)}).
    \]
\end{lemma}
\begin{proof}
    This is again true by Finner's inequality with weights $1/\Delta$ on every edge.
\end{proof}

\begin{lemma}[{\cite[Lemma~5.4]{BGLZ2017}}]
    \label{lem:convex}
    Lef $f, g$ be convex nondecreasing functions on $[0, \infty)$ and let $a > 0$. The minimum of $x + y$ over the region $\{ x ,y \ge 0 : f(x) + g(y) \ge a \}$ is attained at either $x = 0$ or $y = 0$.
\end{lemma}

\begin{proof}[Proof of Theorem~\ref{thm:lower-bd}(b)]
    By applying similar methods as before, we get
    \[
        \Phi_p(H, \delta) \ge \inf \{ x_1 y_1 + x_2 y_2 + z : g_H(x_1 , x_2 , y_1 , y_2) + z^{\mathsf{e}(H) / \Delta} = 1 + \delta \}.
    \]
    By applying Lemma~\ref{lem:convex} to $x_1$ and $z$ (fixing all other variables), we see that the infimum is obtained when either $x_1 = 0$ or $z = 0$. Repeating this argument with $x_2$ in the place of $x_1$, we see that either $x_1 = x_2 = 0$ or $z = 0$. Since $\mathsf{e}(H) / \Delta = \mathsf{v}(H) / 2$, this implies part (b) of Theorem~\ref{thm:lower-bd}.
\end{proof}

\section{Gaps between upper and lower bounds}\label{sec:gap}

\subsection{Comparing $f_H$ and $g_H$}
The gap between $f_H$ and $g_H$ is caused by the difference between $a_S$ and $A_S$. Indeed, if $a_S = A_S$ and $b_S =  B_S$ for all $S$, then $f_H = g_H$ and we can solve the variational problem.The following proposition gives information on the range of $a_S, b_S$.

\begin{proposition}\label{prop:asbd}
    For each $S \in \mathcal{S}_H$, let $F \in \mathcal{F}_H$ be the bipartite digraph associated with $S$ and $T = \mathsf{N}(S)$. Further, let $\tau(G) = \tau(\mathsf{V}, \mathsf{E})$ denote the minimum size of a vertex cover on $G = (\mathsf{V}, \mathsf{E})$. We have
    \[
         \mathsf{v}^+(S) \le a_S \le \tau(S \cup T, \mathsf{E}(S, T)) \le \min(\mathsf{v}^+(S) + \mathsf{v}^{\pm} (S), A_S)
    \]\[
        \mathsf{v}^-(S) \le b_S \le \tau(S \cup T, \mathsf{E}(T, S)) \le \min(\mathsf{v}^-(S) + \mathsf{v}^{\pm} (S), B_S)
    \]
\end{proposition}

\begin{proof}
    We only prove the first set of inequalities. The last inequality is true since $\mathsf{V}^+(S) \cup \mathsf{V}^{\pm} (S)$ and $\mathsf{N}^+(S)$ are vertex covers of $\mathsf{E}(S, \mathsf{N}^+(S))$ (recall that $A_S = |\mathsf{N}^+(S)|$). The first inequality is due to the fact that for any maximum fractional matching of $F$, the weights of edges in $\mathsf{E}(S, T)$ starting from $S$ must equal $\mathsf{v}^+(S)$. To see the second inequality, suppose $w_F$ is a maximum fractional matching on $F$. Note that $w_F$ with the domain restricted to $E(S, T)$ is a (not necessarily maximum) fractional matching on the subgraph $(\mathsf{V}, \mathsf{E}(S, T))$ of $G$. The next two lemmas complete our proof.
\end{proof}

\begin{lemma}[{\cite[Exercise~7.2.1]{LP2009}}]
    \label{lem:frac-matching}
    For any bipartite digraph $G$, the maximum fractional matching number is equal to the maximum matching number.
\end{lemma}

\begin{lemma}[K\"onig's Theorem, {\cite[Theorem~1.1.1]{LP2009}}]
        \label{lem:konig}
        For any bipartite digraph $G$, the maximum matching number is equal to the size of the minimum vertex cover.
\end{lemma}

For proofs, see the referenced texts.

Even if $a_S < A_S$, there is still a chance that we get matching upper and lower bounds. Some possible cases are described in the following remark. Justification for each of them is straightforward.

\begin{remark} \quad \label{rmk}
    \begin{enumerate}
        \item If $a_S = A_S$ for all $S \in \mathcal{S}_H$ and $G(H, \delta)$ is obtained when $y_2 = 1$, then $F(H, \delta) = G(H, \delta)$. Similar for $b_S = B_S$.
        \item If $G(H, \delta)$ is obtained when $y_1 = y_2 = 1$, then we have $F(H, \delta) = G(H, \delta)$.
    \end{enumerate}
\end{remark}

Thus, it is important to know when the lower bound is obtained, and whether $y_i = 1$ in those cases. To this end, we show the following propositions.

\begin{proposition}\label{prop:x1=x2}
    $G(H, \delta)$ is obtained when $x_1 = x_2$. Hence, finding $G(H, \delta)$ is reduced to a $2$-dimensional variational problem, i.e.
    \[
        G(H, \delta) = \inf_{\substack{0 \le x \\ 0 \le y_1 , y_2 \\ y_1 \vee y_2 = 1}} \{ x(y_1 + y_2) : g_H (x, x, y_1 , y_2) = 1 + \delta \}.
    \]
\end{proposition}

\begin{proof}
    Suppose $x_1 > x_2$. We shall prove that we can always increase $g_H (x_1, x_2, y_1, y_2)$ while keeping $x_1 y_1 + x_2 y_2$ constant. Indeed, note that
    \[
        g_H(x_1 , x_2 , y_1 , y_2) = \sum_{S \in \mathcal{S}_H} x_1^{\mathsf{v}^+(S)} x_2^{\mathsf{v}^-(S) + \mathsf{v}^{\pm}(S)} y_1^{a_S} y_2^{b_S}.
    \]
    By Proposition~\ref{prop:asbd}, $\mathsf{v}^-(S) + \mathsf{v}^{\pm}(S) \ge b_S$. Thus, under constant $x_2 y_2$, we can increase $g_H$ by increasing $x_2$ and decreasing $y_2$.
\end{proof}

\begin{proposition}\label{prop:F-2G}
    \[
        G(H, \delta) \le F(H, \delta) \le 2 G(H, \delta).
    \]
\end{proposition}

\begin{proof}
    The first inequality has already been established. Now suppose $g_H(x_1, x_2, y_1, y_2) = 1 + \delta$ with $x_1 y_1 + x_2 y_2 = G(H, \delta)$. By Proposition~\ref{prop:x1=x2}, we may assume $x_1 = x_2$ and $y_1 \vee y_2 = 1$. Note that
    \[
        f_H(x_1 , x_2 , 1, 1) = g_H (x_1 , x_2 , 1 , 1) \ge g_H(x_1 , x_2 , y_1 , y_2) = 1 + \delta.
    \]
    This implies $x_1 + x_2 \ge F(H, \delta)$. Therefore,
    \[
        F(H, \delta) \le x_1 + x_2 = 2x_1 \le 2 x_1 (y_1 + y_2) = 2G(H, \delta).
    \]
\end{proof}
\begin{proposition}\label{prop:2s}
    If $2a_S \ge |S|$ for all $S$, then the $G(H, \delta)$ is obtained when $y_1 = 1$. Similarly, if $2b_S \ge |S|$ for all $S$, then the $G(H, \delta)$ is obtained when $y_2 = 1$.
\end{proposition}

\begin{proof}
    Suppose $2a_S \ge |S|$ for all $S$ and $0 \le y_1 < y_2 = 1$. It suffices to show that under constant $x(1 + y_1)$, $g_H$ increases as $y_1$ increases (Here, $x = x_1 = x_2$ by Proposition~\ref{prop:x1=x2}). Now note that every term of $g_H$ is of the form $x^{|S|} y_1^{a_S}$, where $c > 0$. Thus, the problem is reduced to showing that $x^{\alpha}y_1$ as $y_1$ increases ($x(1 + y_1) = C$ is constant) for all $1 \le \alpha \le 2$ (note that $a_S \le |S|$ by Proposition~\ref{prop:asbd}). This can be justified by plugging $y_1 = C/x - 1$ and differentiating with respect to $x$.
\end{proof}
\subsection{Examples of matching bounds}

\begin{example}[Stars]
Suppose $H$ is a star with one $\Delta$-degree vertex $v$ and $\Delta$ $1$-degree vertices $u_1 , \dots , u_\Delta$ adjacent to $v$. If all vertices are pointing away from $v$, then
\[
    f_H = 1 + x_1 y_1^{\Delta}, \quad g_H = 1 + x_1 y_1.
\]
For both $f_H = 1 + \delta$ and $g_H = 1 + \delta$, $x_1 y_1 + x_2 y_2$ is minimized when $x_1 = \delta$, $y_1 = 1$, and $x_2 = y_2 = 0$. As such, we have
\[
\frac{\Phi_{n, p}(H, \delta)}{n^2 p^\Delta \log(1/p)}  = \delta + o(1).
\]
An identical process can be repeated when all the edges point to $v$.

An analogous argument can be used to show that when $v$ has both in and out degrees,
\[
\frac{\Phi_{n, p}(H, \delta)}{n^2 p^\Delta \log(1/p)}  = 2\delta + o(1).
\]
\end{example}

\begin{example}[Triangles]\label{ex:triangle}

Up to isomorphism, there are two directed triangles. In both cases, we have
\[
\frac{\Phi_{n, p}(H, \delta)}{n^2 p^\Delta \log(1/p)} = \min\left\{\delta^{2/3}, \frac{2}{3} \delta \right\}  + o(1).
\]
To see this, first consider the following triangle $H$.

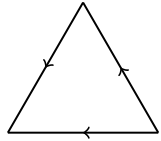
\begin{center}
    \begin{tikzpicture}[
        midarrow/.style={postaction={decorate,decoration={markings,mark=at position 0.5 with {\arrow{>}}}}}
      ]
        \coordinate (A) at (0,0);
        \coordinate (B) at (2,0);
        \coordinate (C) at (1,{sqrt(3)});
        
        \draw[midarrow] (B) -- (A);
        \draw[midarrow] (B) -- (C);
        \draw[midarrow] (C) -- (A);
      \end{tikzpicture}
\end{center}

We have
\[
    f_H(x_1 , x_2 , y_1 , y_2) = 1 + x_1 y_1^2 + x_2 y_2^2 + (x_1 \wedge x_2) y_1 y_2 
\]
\[
    g_H(x_1 , x_2 ,  y_1 , y_2) = 1 + x_1 y_1 + x_2 y_2 + (x_1 \wedge x_2) y_1 y_2 
\]
$G(H, \delta)$ is obtained when $y_1 = y_2 = 1$, implying that the upper and lower bounds must match. Further, $G(H, \delta) = 2/3 \delta$, given when $x_1 = x_2 = \delta/3$ and $y_1 = y_2 = 1$. A similar process can be done for the other triangle, or it can also be viewed as a special case of the next example.
\end{example}

\begin{example}[Balanced digraphs]\label{ex:balanced}
    Let $H$ be a balanced digraph (i.e. all vertices have equal number of in/out edges). In this case, we can apply Proposition~\ref{prop:2s} to solve to the variational problem. Indeed, suppose $S \in \mathcal{S}_H$ and let $F \in \mathcal{F}_H$ be the associated bipartite digraph. The maximum fractional matching $w_F \equiv 1 / \Delta$ gives $\sum_{e \in \mathsf{E}(S, T)} w_F (e) = \sum_{e \in \mathsf{E}(T, S)} w_F(e) = |S|/2$. This implies $a_S, b_S \ge |S| / 2$, satisfying the conditions of Proposition~\ref{prop:2s}. Therefore, the lower bound is obtained when $y_1 = y_2 = 1$. Thus, we have $G(H, \delta) = F(H, \delta)$.

    In fact, we can weaken the condition to only require vertices with maximum degree being balanced.
\end{example}

\begin{example}[$0 < y_1 < 1$]

Consider a bipartite digraph with $|H^*| = k$ and $|T| = k+1$. The following is an example of when $k = 3$.

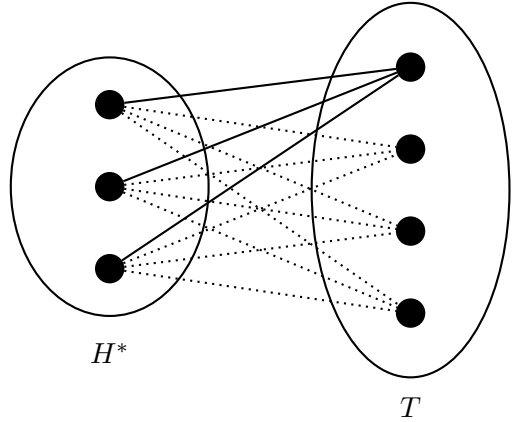
\begin{center}
    \begin{tikzpicture}[thick,
        every node/.style={draw,circle},
        fsnode/.style={fill=black},
        ssnode/.style={fill=black},
        every fit/.style={ellipse,draw,inner sep=-2pt,text width=2cm}
      ]
      \begin{scope}[start chain=going below,node distance=7mm]
      \foreach \i in {1,2,...,3}
        \node[fsnode,on chain] (f\i) {};
      \end{scope}
      
      \begin{scope}[xshift=4cm,yshift=0.5cm,start chain=going below,node distance=7mm]
      \foreach \i in {4,5,...,7}
        \node[ssnode,on chain] (s\i) {};
      \end{scope}
      
      \node [black,fit=(f1) (f3),label=below:$H^*$] {};
      \node [black,fit=(s4) (s7),label=below:$T$] {};
      
      \foreach \i in {1,2,...,3}
        \draw[black] (f\i) -- (s4);

        \foreach \i in {1,2,...,3}
        \foreach \j in {5,6,...,7}
        \draw[dotted] (f\i) -- (s\j);
      \end{tikzpicture}
\end{center}

Here, solid lines denote edges going from $H^*$ to $T$, and dotted lines from $T$ to $H^*$. For larger $k$, consider an analogous digraph where one vertex in $T$ as only in-edges and all other vertices have only out-edges.

Clearly, $\mathsf{V}^+(H^*) = \mathsf{V}^-(H^*) = \emptyset$ and $H^*$ is an independent set. Further, $a_S = 1$ and $b_S = |S|$ for any nonempty $S \in \mathcal{S}_H$ (so $a_S \le b_S$). As such, it is always optimal to choose $y_2 = 1$ for the lower bound. Since $A_S = a_S = 1$ always, this implies that we get matching upper and lower bounds. In particular, this means that 
\[
    \lim_{n \to \infty} \frac{\Phi_{n, p}(H, \delta)}{n^2 p^\Delta \log (1/p)} = \inf_{\substack{0 \le x \\ 0 \le y_1 \le 1}} \left\{ x(1 + y_1) : 1 + y_1[(1 + x)^{k} - 1] = 1 + \delta \right\}.
\]
This gives examples of solutions where $0 < y_1 < 1$. For instance, when $k = 3$ and $\delta = 100$, we have $y_1 \approx 0.691$ (by numerical calculations).
\end{example}

\subsection{Gaps between upper and lower bounds}
\begin{example}
    Consider a bipartite digraph with $|H^*| = k+1$ and $|T| = k+2$. The following is an example of when $k = 3$.

    
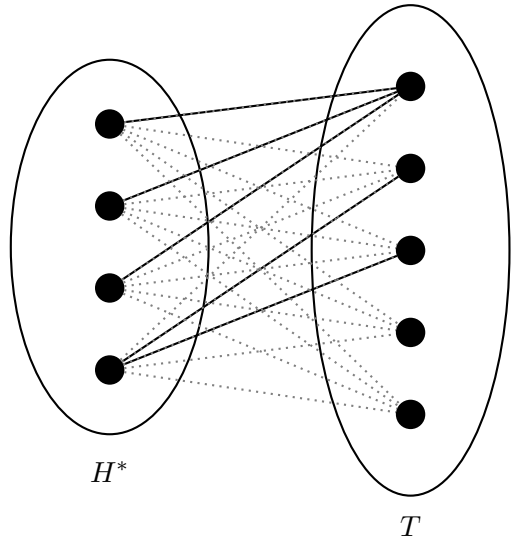
\begin{center}
        \begin{tikzpicture}[thick,
            every node/.style={draw,circle},
            fsnode/.style={fill=black},
            ssnode/.style={fill=black},
            every fit/.style={ellipse,draw,inner sep=-2pt,text width=2cm}
          ]
          \begin{scope}[start chain=going below,node distance=7mm]
          \foreach \i in {1,2,...,4}
            \node[fsnode,on chain] (f\i) {};
          \end{scope}
          
          \begin{scope}[xshift=4cm,yshift=0.5cm,start chain=going below,node distance=7mm]
          \foreach \i in {1,2,...,5}
            \node[ssnode,on chain] (s\i) {};
          \end{scope}
          
          \node [black,fit=(f1) (f4),label=below:$H^*$] {};
          \node [black,fit=(s1) (s5),label=below:$T$] {};
          
          \foreach \i in {1,2,...,3}
            \draw[black] (f\i) -- (s1);
          \draw[black] (f4) -- (s2);
          \draw[black] (f4) -- (s3);  
          \foreach \i in {1,2,...,4}
          \foreach \j in {1,2,...,5}
          \draw[gray, dotted] (f\i) -- (s\j);
          \end{tikzpicture}
    \end{center}

Again, solid lines denote edges going from $H^*$ to $T$, and dotted lines from $T$ to $H^*$. For larger $k$, consider an analogous digraph where $v_1 , \dots , v_k \in H^*$ each have one out-edge to $u_1 \in T$, while $v_{k+1}$ has out-edges to $u_2$ and $u_3$. It is straightforward to check that the infimum for $f_H$ and $g_H$ are obtained when $y_2 = 1$, and that
\[
    f_H(x_1 , x_2 , y_1 , 1) = [1 + ((1 + (x_1 \wedge x_2))^k - 1)y_1](1 + (x_1 \wedge x_2) y_1^2)
\]
\[
    g_H(x_1 , x_2 , y_1 , 1) = [1 + ((1 + (x_1 \wedge x_2))^k - 1)y_1](1 + (x_1 \wedge x_2) y_1)
\]

Since $f_H > g_H$ for for $0 < y_1 < 1$, we can expect to see a gap between bounds. Indeed, when $k = 5$ and $\delta = 10000$, we get $F(H, \delta) \approx 7.283$ while $G(H, \delta) \approx 7.031$.
\end{example}


\begin{thebibliography}{10}

    \bibitem{A2020}
    F.~Augeri.
    \newblock Nonlinear large deviation bounds with applications to {W}igner matrices and sparse {E}rd{\H{o}}s-{R}\'{e}nyi graphs.
    \newblock {\em Ann. Probab.}, 48(5):2404--2448, 2020.
    
    \bibitem{JG2009}
    J.~r. Bang-Jensen and G.~Gutin.
    \newblock {\em Digraphs}.
    \newblock Springer Monographs in Mathematics. Springer-Verlag London, Ltd., London, second edition, 2009.
    \newblock Theory, algorithms and applications.
    
    \bibitem{JG2018}
    J.~r. Bang-Jensen and G.~Gutin, editors.
    \newblock {\em Classes of directed graphs}.
    \newblock Springer Monographs in Mathematics. Springer, Cham, 2018.
    
    \bibitem{BB2023}
    A.~Basak and R.~Basu.
    \newblock Upper tail large deviations of regular subgraph counts in {E}rd{\H{o}}s-{R}\'{e}nyi graphs in the full localized regime.
    \newblock {\em Comm. Pure Appl. Math.}, 76(1):3--72, 2023.
    
    \bibitem{BGLZ2017}
    B.~B. Bhattacharya, S.~Ganguly, E.~Lubetzky, and Y.~Zhao.
    \newblock Upper tails and independence polynomials in random graphs.
    \newblock {\em Adv. Math.}, 319:313--347, 2017.
    
    \bibitem{BD2021}
    S.~Bhattacharya and A.~Dembo.
    \newblock Upper tail for homomorphism counts in constrained sparse random graphs.
    \newblock {\em Random Structures Algorithms}, 59(3):315--338, 2021.
    
    \bibitem{CD2016}
    S.~Chatterjee and A.~Dembo.
    \newblock Nonlinear large deviations.
    \newblock {\em Adv. Math.}, 299:396--450, 2016.
    
    \bibitem{CV2011}
    S.~Chatterjee and S.~R.~S. Varadhan.
    \newblock The large deviation principle for the {E}rd{\H{o}}s-{R}\'{e}nyi random graph.
    \newblock {\em European J. Combin.}, 32(7):1000--1017, 2011.
    
    \bibitem{CD2020}
    N.~Cook and A.~Dembo.
    \newblock Large deviations of subgraph counts for sparse {E}rd{\H{o}}s-{R}\'{e}nyi graphs.
    \newblock {\em Adv. Math.}, 373:107289, 53, 2020.
    
    \bibitem{cook2022typical}
    N.~A. Cook and A.~Dembo.
    \newblock Typical structure of sparse exponential random graph models, 2022.
    
    \bibitem{CDP2023}
    N.~A. Cook, A.~Dembo, and H.~T. Pham.
    \newblock Regularity method and large deviation principles for the {E}rd{{\H{o}}}s--{R}{\'e}nyi hypergraph, 2023.
    
    \bibitem{DK2012}
    B.~DeMarco and J.~Kahn.
    \newblock Upper tails for triangles.
    \newblock {\em Random Structures Algorithms}, 40(4):452--459, 2012.
    
    \bibitem{DL2018}
    A.~Dembo and E.~Lubetzky.
    \newblock A large deviation principle for the {E}rd{\H{o}}s-{R}\'{e}nyi uniform random graph.
    \newblock {\em Electron. Commun. Probab.}, 23:Paper No. 13, 2018.
    
    \bibitem{DZ2010}
    A.~Dembo and O.~Zeitouni.
    \newblock {\em Large deviations techniques and applications}, volume~38 of {\em Stochastic Modelling and Applied Probability}.
    \newblock Springer-Verlag, Berlin, 2010.
    \newblock Corrected reprint of the second (1998) edition.
    
    \bibitem{E2018}
    R.~Eldan.
    \newblock Gaussian-width gradient complexity, reverse log-{S}obolev inequalities and nonlinear large deviations.
    \newblock {\em Geom. Funct. Anal.}, 28(6):1548--1596, 2018.
    
    \bibitem{F1992}
    H.~Finner.
    \newblock A generalization of {H}\"{o}lder's inequality and some probability inequalities.
    \newblock {\em Ann. Probab.}, 20(4):1893--1901, 1992.
    
    \bibitem{G2021}
    B.~Gunby.
    \newblock Upper tails of subgraph counts in sparse regular graphs, 2021.
    
    \bibitem{HMS2022}
    M.~Harel, F.~Mousset, and W.~Samotij.
    \newblock Upper tails via high moments and entropic stability.
    \newblock {\em Duke Math. J.}, 171(10):2089--2192, 2022.
    
    \bibitem{JR2002}
    S.~Janson and A.~Ruci\'{n}ski.
    \newblock The infamous upper tail.
    \newblock {\em Random Structures Algorithms}, 20(3):317--342, 2002.
    \newblock Probabilistic methods in combinatorial optimization.
    
    \bibitem{LZ2021}
    Y.~P. Liu and Y.~Zhao.
    \newblock On the upper tail problem for random hypergraphs.
    \newblock {\em Random Structures Algorithms}, 58(2):179--220, 2021.
    
    \bibitem{L2012}
    L.~Lov\'{a}sz.
    \newblock {\em Large networks and graph limits}, volume~60 of {\em American Mathematical Society Colloquium Publications}.
    \newblock American Mathematical Society, Providence, RI, 2012.
    
    \bibitem{LP2009}
    L.~Lov\'{a}sz and M.~D. Plummer.
    \newblock {\em Matching theory}.
    \newblock AMS Chelsea Publishing, Providence, RI, 2009.
    \newblock Corrected reprint of the 1986 original [MR0859549].
    
    \bibitem{LZ2015}
    E.~Lubetzky and Y.~Zhao.
    \newblock On replica symmetry of large deviations in random graphs.
    \newblock {\em Random Structures Algorithms}, 47(1):109--146, 2015.
    
    \bibitem{LZ2017}
    E.~Lubetzky and Y.~Zhao.
    \newblock On the variational problem for upper tails in sparse random graphs.
    \newblock {\em Random Structures Algorithms}, 50(3):420--436, 2017.
    
    \bibitem{M2023}
    M.~Markering.
    \newblock The large deviation principle for inhomogeneous {E}rd{{\H{o}}}s--{R}{\'e}nyi random graphs.
    \newblock {\em J. Theoret. Probab.}, 36(2):711--727, 2023.
    
    \bibitem{V2001}
    V.~H. Vu.
    \newblock A large deviation result on the number of small subgraphs of a random graph.
    \newblock {\em Combin. Probab. Comput.}, 10(1):79--94, 2001.
    
    \bibitem{Z2017}
    Y.~Zhao.
    \newblock On the lower tail variational problem for random graphs.
    \newblock {\em Combin. Probab. Comput.}, 26(2):301--320, 2017.
    
    \end{thebibliography}
\end{document}